\documentclass[a4paper,12pt,reqno,english]{amsart}
\usepackage[utf8]{inputenc}
\usepackage[T1]{fontenc}
\usepackage{babel}
\usepackage{float}

\usepackage[bookmarks = true, colorlinks, citecolor = blue, urlcolor = blue]{hyperref}
\usepackage{thmtools}
\usepackage[capitalize, nameinlink]{cleveref}
\usepackage[titletoc, toc, title]{appendix}

\hypersetup{pdfstartview=}

\usepackage{amssymb}
\usepackage{tikz}
\usepackage{tikz-cd}
\usetikzlibrary{arrows.meta, positioning}
\usetikzlibrary{shapes}

\usepackage{ytableau}
\ytableausetup{centertableaux, boxsize = 1.4 em}

\numberwithin{equation}{section}
\numberwithin{figure}{section}

\theoremstyle{plain}
\newtheorem{theorem}{Theorem}[section]

\theoremstyle{plain}
\newtheorem*{theorem*}{Theorem}

\theoremstyle{plain}
\newtheorem{proposition}[theorem]{Proposition}

\theoremstyle{plain}
\newtheorem{lemma}[theorem]{Lemma}

\theoremstyle{plain}
\newtheorem{corollary}[theorem]{Corollary}

\theoremstyle{definition}
\newtheorem{definition}[theorem]{Definition}

\theoremstyle{definition}

\theoremstyle{definition}
\newtheorem{example}[theorem]{Example}

\theoremstyle{remark}
\newtheorem{remark}[theorem]{Remark}

\theoremstyle{definition}
\newtheorem{question}[theorem]{Question}

\addto\extrasenglish{
  
}

\newcommand{\Uqg}{U_q(\mathfrak{g})}
\newcommand{\OqG}{\mathcal{O}_q(G)}
\newcommand{\CqG}{C^*_q(G)}

\newcommand{\bbN}{\mathbb{N}}

\newcommand{\lieg}{\mathfrak{g}}

\newcommand{\calB}{\mathcal{B}}

\newcommand{\calO}{\mathcal{O}}

\newcommand{\wt}{\mathrm{wt}}

\newcommand{\sfR}{\mathsf{R}}

\newcommand{\dwei}{\mathsf{P}^+}

\newcommand{\hrg}{\Gamma_{\mathfrak{g}}}

\newcommand{\wgc}{\mathsf{s}}

\newcommand{\Rend}{\sfR}

\newcommand{\bruhat}{B(W)}
\newcommand{\bruhatL}{B_L(W)}
\newcommand{\bruhatR}{B_R(W)}

\newcommand{\coloring}{\mathsf{c}}

\newcommand{\suppr}{\mathrm{supp}_{\mathrm{r}}}
\newcommand{\suppw}{\mathrm{supp}_{\mathrm{w}}}

\newcommand{\lkey}{K_-}
\newcommand{\rkey}{K_+}

\begin{document}

\title[Results on the higher-rank graphs associated to semisimple Lie algebras]{Some results on the higher-rank graphs associated \\ to crystals of semisimple Lie algebras}

\author{Marco Matassa}

\address{Department of Computer Science, OsloMet – Oslo Metropolitan University, Oslo, Norway}

\email{marco.matassa@oslomet.no}

\begin{abstract}
In this paper we continue the study of the higher-rank graphs associated to finite-dimensional complex semisimple Lie algebras, introduced by the author and R. Yuncken, whose construction relies on Kashiwara's theory of crystals.
First we prove that the Bruhat graphs of the corresponding Weyl groups, both weak and strong, can be embedded into the higher-rank graphs as colored graphs.
Next, specializing to Lie algebras of type $A$, we connect some aspects of the construction of the higher-rank graphs with some well-known notions in combinatorics, most notably the keys of Lascoux and Schützenberger.
\end{abstract}

\maketitle

\section{Introduction}

\subsection{Background}

The notion of \emph{crystals} associated to a complex semisimple Lie algebra $\lieg$, originally due to Kashiwara \cite{kashiwara1, kashiwara2}, is by now an established tool in representation theory and combinatorics (see for instance \cite{hong-kang, bump-schilling} for textbook accounts).
For each $\lieg$-module $V$, it allows the construction of a corresponding crystal $\calB$, which is a combinatorial object that encodes all the essential information about such a module.
One way to construct such $\lieg$-crystals is via the \emph{quantized enveloping algebra} $\Uqg$, where they arise from its modules in a certain limit $q \to 0$ of the deformation parameter.
In particular, this theory can be applied to produce a crystal analogue of the \emph{quantized coordinate ring} $\OqG$ of the corresponding Lie group $G$, which has been investigated from this perspective in \cite{kashiwara-weyl}.

A different approach to the construction of a $q \to 0$ limit for the quantized coordinate ring $\OqG$ was proposed in \cite{matassa-yuncken}.
It is also based on the category of $\lieg$-crystals, but keeps track of the $*$-structure corresponding to the compact real form of $\lieg$, which is necessary for the analytic picture in terms of $C^*$-algebras.
One of the main results there was the connection with the notion of \emph{higher-rank graphs}, introduced by Kumjian and Pask in \cite{kumjian-pask}.
These allow the construction of $C^*$-algebras which generalize the \emph{graph $C^*$-algebras} corresponding to directed graphs.
We review various results related to this perspective below.

Let us write $\CqG$ for the $C^*$-completion of $\OqG$ with deformation parameter $q \in (0, 1)$.
It has been known since \cite{woronowicz} that the $C^*$-algebras $C^*_q(SU(2))$ are isomorphic for all values of $q \in [0, 1)$ (where the case $q = 0$ has to be interpreted appropriately).
It was observed in \cite{hong-szymanski} that these can be described as graph $C^*$-algebras, where the underlying graph structure is most transparent in the case $q = 0$.
However, it was also observed that $\CqG$ does not admit a description of this type for (simple) $G$ of rank larger than one, because of known restrictions imposed by the so-called gauge action.
On the other hand, the $C^*$-algebras corresponding to higher-rank graphs do not have such restrictions, and hence are natural candidates for the general description of $\CqG$.
This perspective was first explored by Giselsson for $G = SU(3)$ in \cite{giselsson} (and appeared later in print), where it was shown that $C^*_q(SU(3))$ can be modeled as a higher-rank graph $C^*$-algebra of rank two. Furthermore, such $C^*$-algebras are isomorphic for all $q \in [0, 1)$, as in the case of $SU(2)$.
Based on this insight, a general construction was given in \cite{matassa-yuncken}, which can be applied to any Lie group $G$ corresponding to a finite-dimensional complex semisimple Lie algebra $\lieg$: upon representing $\OqG$ on a Hilbert space by a faithful $*$-representation, it produces a $*$-algebra $\calO_0(G)$ as a certain $q \to 0$ limit of $\OqG$, whose $C^*$-completion $C^*_0(G)$ was shown to be isomorphic to a concretely defined higher-rank graph $C^*$-algebra.
It is conjectured that $C^*_0(G)$ should be isomorphic to $\CqG$ for any $q \in (0, 1)$, but this is not known in general (on the other hand, it is known that the $C^*$-algebras $\CqG$ are isomorphic for all $q \in (0, 1)$ by \cite{q-independence}).

Informally, we can think of a higher-rank graph as a generalization of a directed colored graph, with vertices connected by paths and with rank given by the number of colors (for its precise definition see \cref{def:higher-rank}).
As shown in \cite{matassa-yuncken}, to any finite-dimensional complex semisimple Lie algebra $\lieg$ we can associate a higher-rank graph $\hrg$ with rank equal to that of $\lieg$, and whose construction makes use of the corresponding category of $\lieg$-crystals.
The following general properties of the higher-rank graphs $\hrg$ are proven in the cited paper:
\begin{itemize}
\item the \emph{Weyl group} $W$ of $\lieg$ embeds into the vertex set of $\hrg$,
\item this set admits a \emph{partial order} with unique maximal and minimal elements, corresponding to the Weyl group elements of maximal and minimal lengths,
\item the paths in $\hrg$ are \emph{compatible} with this partial order.
\end{itemize}
On the other hand, not much is known about which paths arise in $\hrg$, even in the case of vertices corresponding to Weyl group elements.
Recall that the \emph{Bruhat graph} of $W$ is a certain canonical directed graph attached to $W$, having its elements as vertices.
Since $W$ embeds into the vertex set of $\hrg$, a natural question arises: what is the relation between the Bruhat graph of $W$ and the paths in $\hrg$ between vertices corresponding to $W$?
This provides the main motivation for this paper, which aims to elucidate this structure.

\subsection{Results}

The first aspect we investigate is whether we can embed the various Bruhat graphs associated to $W$ into the higher-rank graph $\hrg$.
We begin by considering the graphs corresponding to the right and left \emph{weak} Bruhat orders on $W$, which can be seen as colored graphs in a natural way.
In \cref{thm:embedding-rightweak} we show that the right weak Bruhat graph has a \emph{unique} embedding, as a colored graph, into the \emph{skeleton} of $\hrg$ (where the skeleton is a colored graph that can be canonically attached to any higher-rank graph).
The situation is different for the left weak Bruhat graph, which does not admit such an embedding in general (however we note that the difference of left versus right is largely a matter of conventions).
On the other hand, in \cref{thm:embedding-bruhat} we prove that we can embed the (strong) Bruhat graph into the higher-rank graph $\hrg$, rather than just its skeleton.
This is true for any choice of \emph{compatible coloring} for the Bruhat graph, as described in \cref{def:compatible-coloring}, which shows that many such embeddings are possible.
We illustrate and discuss these results for $\lieg = A_2 = \mathfrak{sl}_3$.

Next, we investigate some specific aspects of the construction for $\lieg = A_r$, with the goal of making connections with some well-known notions in combinatorics.
One important ingredient in the construction of $\hrg$ is the notion of \emph{right ends} of elements of a crystal, introduced in \cite{matassa-yuncken}.
For any $\lieg$, the vertex set of $\hrg$ is constructed as the right ends of the elements of the crystal $\calB(\rho)$ corresponding to the regular weight $\rho = \omega_1 + \cdots + \omega_r$.
For $\lieg = A_r$, it is well-known that elements of crystals can be put in one-to-one correspondence with semistandard Young tableaux (a result originally due to Kashiwara-Nakashima \cite{kashiwara-nakashima}).
In light of this identification, taking the right ends of crystal elements should correspond to some operation on such tableaux. In \cref{thm:keys-ends} it is shown that the right ends coincide with the (left) \emph{keys} of Lascoux and Schützenberger \cite{keys}.
Along the way, we also discuss the relation between the notion of \emph{crystal braiding} and the jeu de taquin slides for skew tableaux.

Finally we offer some brief discussion on the (tentative) connection between right ends and a \emph{generalization of keys} introduced in \cite[Definition 5.2]{lenart-involution}.
The latter notion can be defined for a Lie algebra $\lieg$ of arbitrary type, and reduces in type $A$ to the one introduced by Lascoux and Schützenberger in an appropriate sense.
Such a key can be defined as a map from a crystal to the Weyl group of $\lieg$, and can be characterized by some axioms given in \cite[Section 3]{hersh-lenart}.
This description makes it clear that the right ends \emph{do not} coincide with such keys in general, since not all right ends can be put in correspondence with Weyl group elements.
However we expect that such keys should coincide with the right ends for many elements of the crystal $\calB(\rho)$, and we discuss this explicitly in the case of $\lieg = C_2$.

\subsection{Organization}

The paper is organized as follows.
In \cref{sec:preliminaries} we recall some basic facts concerning crystals and the Bruhat order.
In \cref{sec:higher-rank} we review the construction of the higher-rank graphs $\hrg$ from \cite{matassa-yuncken}, highlighting the main aspects relevant for this paper.
In \cref{sec:technical-results} we prove the main technical results needed for the embeddings of the Bruhat graphs into $\hrg$.
These are discussed in \cref{sec:embedding}, where we distinguish between the right weak Bruhat order, the left weak Bruhat order and the (strong) Bruhat order.
We also comment on these results and illustrate them in the case of $\lieg = A_2$.
In \cref{sec:right-ends} we specialize to $\lieg = A_r$ and reinterpret various aspects of the construction of $\hrg$.
In particular, we show how the crystal braiding can be computed in terms of jeu de taquin slides, and how the right ends can be be interpreted as left keys.
Finally in \cref{sec:beyond-A} we briefly discuss the (tentative) connection with the mentioned generalization of keys, illustrating this for $\lieg = C_2$.

\subsection*{Acknowledgments}

It is a pleasure to thank Robert Yuncken for the great collaboration in \cite{matassa-yuncken}, as well as for interesting discussions related to the contents of this paper.
We would also like to thank Makoto Yamashita for a stimulating conversation, that sparked the author's interest into proving some of the results discussed here.

\section{Some preliminaries}
\label{sec:preliminaries}

In this section we review some standard material concerning crystals associated to a finite-dimensional complex semisimple Lie algebra $\lieg$, and the Bruhat order on the corresponding Weyl group $W$.
In particular, we emphasize the action of the Weyl group on crystals.

\subsection{Crystals}

We begin by fixing some standard notations for a finite-dimensional complex semisimple Lie algebra $\lieg$.
We write $I = \{ 1, \ldots, r \}$ for the index set corresponding to the rank $r$.
The simple roots and fundamental weights are denoted by $\{ \alpha_i \}_{i \in I}$ and $\{ \omega_i \}_{i \in I}$, respectively.
We write $(\cdot, \cdot)$ for the usual invariant bilinear form on the dual of the Cartan subalgebra.
The coroots are denoted by $\alpha^\vee = 2 \alpha / (\alpha, \alpha)$ and we recall the relation $(\omega_i, \alpha_j^\vee) = \delta_{i j}$.
We write $\dwei$ for the dominant weights and $\preceq$ for the usual partial order on weights (meaning that $\mu \prec \lambda$ if $\lambda - \mu$ is an element of the positive root lattice).

For the precise definition of a $\lieg$-crystal we refer to \cite[Section 4.5]{hong-kang} (all the crystals we consider here are \emph{semiregular}), while below we only highlight the most relevant aspects for our purposes.
As part of the definition of a crystal $\calB$, we have maps $\tilde{E}_i, \tilde{F}_i : \calB \to \calB \sqcup \{ 0 \}$ for any $i \in I$, called the \emph{Kashiwara operators}.
These can be seen as a renormalization of the action of the root vectors $E_i, F_i$ for the simple root $\alpha_i$ (see \cite[Section 4.1]{hong-kang} for a precise definition).
Corresponding to these operators, we have the maps $\varepsilon_i, \varphi_i : \calB \to \bbN$ given by
\[
\varepsilon_i(b) = \max \{ k: \tilde{E}_i^k b \in \calB \}, \quad
\varphi_i(b) = \max \{ k: \tilde{F}_i^k b \in \calB \}.
\]
They satisfy the relation $\varphi_i(b) - \varepsilon_i(b) = (\wt(b), \alpha_i^\vee)$, where $\wt(b)$ is the weight of $b \in \calB$.

The \emph{crystal graph} corresponding to $\calB$ is the graph with vertices given by the elements of $\calB$ and with directed edges corresponding to the action of the Kashiwara operators $\{ \tilde{F}_i \}_{i \in I}$.
Given $\lambda \in \dwei$, we write $\calB(\lambda)$ for the unique (up to isomorphism) crystal corresponding to the highest weight $\lambda$.
In this case we use the terminology \emph{connected}, since the corresponding crystal graph turns out to be connected (another standard terminology would be irreducible).

One of the most notable feature of crystals is their tensor product structure.
Given two crystals $\calB$ and $\calB'$, we can form their tensor product $\calB \otimes \calB'$ (which as a set is simply $\calB \times \calB'$). The Kashiwara operators act as follows (with conventions as in \cite{hong-kang}):
\begin{equation}
\label{eq:tensor-rule}
\tilde{E}_i (b \otimes b') = \begin{cases}
\tilde{E}_i b \otimes b' & \varphi_i(b) \geq \varepsilon_i(b') \\
b \otimes \tilde{E}_i b' & \varphi_i(b) < \varepsilon_i(b')
\end{cases}
, \quad
\tilde{F}_i (b \otimes b') = \begin{cases}
\tilde{F}_i b \otimes b' & \varphi_i(b) > \varepsilon_i(b') \\
b \otimes \tilde{F}_i b' & \varphi_i(b) \leq \varepsilon_i(b')
\end{cases}.
\end{equation}
(Let us note that these are the same conventions used in the papers by Kashiwara, but opposite to those used in the book \cite{bump-schilling}, which is going to be relevant later on).

The crystal graph of the tensor product $\calB(\lambda) \otimes \calB(\lambda')$ can be decomposed into connected components, which turn out to correspond to the simple components in the analogous decomposition for the classical (or quantum) modules $V(\lambda) \otimes V(\lambda')$.

There is an action of the Weyl group $W$ on $\lieg$-crystals, which was originally defined in \cite{kashiwara-weyl} (see also \cite[Section 11.3]{bump-schilling}).
Given an element $b \in \calB$ of weight $\wt(b)$, we have
\[
\wgc_i b := \begin{cases}
\tilde{F}_i^k b & k > 0 \\
b & k = 0 \\
\tilde{E}_i^k b & k < 0
\end{cases},
\quad
k = (\wt(b), \alpha_i^\vee).
\]
We note that $\wt(\wgc_i b) = \wt(b) - k \alpha_i^\vee$, which means that $\wt(\wgc_i b) = s_i \wt(b)$, where on the right-hand side we have the ordinary Weyl group action on weights.

Given a dominant weight $\lambda \in \dwei$, the weights in the orbit of $\lambda$ under the Weyl group $W$ are called \emph{extremal}.
In the crystal $\calB(\lambda)$ there is a unique element of weight $w \lambda$ for $w \in W$, which we can then denote by $b_{w \lambda}$.
It follows from the discussion above that $\wgc_i b_{w \lambda} = b_{s_i w \lambda}$.

\subsection{Bruhat order}

We refer to \cite{bjorner-brenti} for any unexplained notions and results.
Let $W$ be a Weyl group with a fixed set of generators $S$, also called the \emph{simple reflections}.
The length function on $W$ is denoted by $\ell$.
We denote by $T := \{ w s w^{-1} : w \in W, \ s \in S \}$ the set of \emph{reflections}.
We recall that the elements of $T$ can be put in one-to-one correspondence with the positive roots of the Lie algebra associated with $W$, and that under this identification the simple reflections correspond to the simple roots, see \cite[Proposition 4.4.5]{bjorner-brenti}.

Let us now recall the notions of Bruhat order and Bruhat graph on $W$.

\begin{definition}
The \emph{Bruhat graph} $\bruhat$ corresponding to $W$ is defined as follows:
\begin{itemize}
\item the vertices are the elements of $W$,
\item we have an edge $u \to w$ if $w = u t$ with $t \in T$ and $\ell(w) > \ell(u)$.
\end{itemize}
We write $w > u$ if there exist elements $w_1, \ldots, w_n$ such that $u \to w_1 \to \cdots \to w_n \to w$.
This is called the (strong) \emph{Bruhat order} on $W$.
\end{definition}

\begin{remark}
The edges in the Bruhat graph can be equivalently defined by the condition $w = t u$ with $t \in T$ and $\ell(u) < \ell(w)$ (that is, by putting the reflection $t$ on the left). This gives the same graph, since we have $u t = (u t u^{-1}) u$ and $u t u^{-1} \in T$.
\end{remark}

The right and left weak versions are defined similarly, but only using simple reflections.

\begin{definition}
The \emph{right weak Bruhat graph} $\bruhatR$ is defined as follows:
\begin{itemize}
\item the vertices are the elements of $W$,
\item we have an edge $u \to w$ if $w = u s_i$ with $s_i \in S$ and $\ell(w) > \ell(u)$.
\end{itemize}
The \emph{left weak Bruhat graph} $\bruhatR$ is defined similarly, with the edges defined by the condition $u \to w$ if $w = s_i u $ with $s_i \in S$ and $\ell(w) > \ell(u)$.
\end{definition}

The right and left weak Bruhat orders are different partial orders on $W$, in general. However they are isomorphic to each other via the map $w \mapsto w^{-1}$.

\section{Higher-rank graph for Lie algebras}
\label{sec:higher-rank}

This section is dedicated to reviewing the construction of the higher-rank graph $\hrg$ from \cite{matassa-yuncken}, which can be defined for any finite-dimensional complex semisimple Lie algebra $\lieg$.
We begin by spelling out the definition of a higher-rank graph, as well as recalling some standard terminology associated to it.
We proceed by reviewing the notions of \emph{Cartan braiding} and \emph{right ends}, which have been introduced in \cite{matassa-yuncken}, as well as their main properties.
Finally we use these two crucial ingredients to explain the construction of $\hrg$.

\subsection{Higher-rank graphs}

The notion of higher-rank graph was introduced in \cite{kumjian-pask} to construct $C^*$-algebras generalizing the more familiar graph $C^*$-algebras.

In the following definition we regard $\bbN^k$ as a category with one object and morphisms corresponding to $k$-tuples of natural numbers, with composition given by their sum.

\begin{definition}
\label{def:higher-rank}
A \emph{higher-rank graph of rank $k$} (or $k$-graph for short) is a countable category $\Lambda$ together with a (degree) functor $d: \Lambda \to \bbN^k$ satisfying the following \emph{factorization property}: for every morphism $f \in \Lambda$ and $k$-tuples $m, n \in \bbN^k$ such that $d(f) = m + n$, there are unique morphisms $g, h \in \Lambda$ with $d(g) = m$ and $d(h) = n$ such that $f = g \circ h$.
\end{definition}

We recall some standard terminology used in this context.
A morphism $f$ with $d(f) = (m_1, \ldots, m_k)$ is also called a \emph{path of color $(m_1, \ldots, m_k)$}.
As a special case, morphisms with $d(f) = (0, \ldots, 1, \ldots, 0)$ with $1$ in position $i$ are called \emph{edges of color $i$}.
The morphisms of degree zero can be identified with the objects of $\Lambda$, and they also referred to as the \emph{vertices}.
The notations $s$ and $r$ are used to denote the source and range of a morphism, respectively.

\begin{example}
Any directed graph can be regarded as a higher-rank graph of rank $1$ in a natural way.
In this case $\Lambda$ is simply the path category corresponding to the graph, with the degree functor $d: \Lambda \to \bbN$ being given by the length of a path.
\end{example}

The \emph{skeleton} of a $k$-graph $\Lambda$ is the directed graph with $k$ colors constructed by taking all vertices and edges of any color of $\Lambda$ (but not general paths).
It is not possible to recover $\Lambda$ from the knowledge of its skeleton, in general.
This requires the extra data of a collection of commuting squares, see \cite{skeleton-graph} (although we are not going to use this structure).

\subsection{Cartan braiding}

In order to describe the higher-rank graphs $\hrg$, we begin by introducing certain morphisms between crystals.
More precisely, they are going to be morphisms defined on the Cartan components of tensor products, in the following sense.

\begin{definition}
Let $\calB(\lambda_1) \otimes \cdots \otimes \calB(\lambda_n)$ be a product of connected crystals.
Then its \emph{Cartan component} is the unique connected component isomorphic to $\calB(\lambda_1 + \cdots + \lambda_n)$.
\end{definition}

Using this terminology, we define the notion of Cartan braiding.

\begin{definition}
Let $\calB(\lambda)$ and $\calB(\lambda')$ be connected crystals.
The \emph{Cartan braiding}
\[
\sigma_{\calB(\lambda), \calB(\lambda')} : \calB(\lambda) \otimes \calB(\lambda') \to \calB(\lambda') \otimes \calB(\lambda)
\]
is the unique morphism of crystals which is an isomorphism between the Cartan components of the two tensor products and is zero otherwise. That is, for $\calB(\mu) \subseteq \calB(\lambda) \otimes \calB(\lambda')$ we have
\[
\sigma_{\calB(\lambda), \calB(\lambda')}(\calB(\mu)) =
\begin{cases}
\calB(\lambda + \lambda') & \mu = \lambda + \lambda' \\
0 & \mathrm{otherwise}
\end{cases}.
\]
\end{definition}

More generally, it can be defined for products of connected crystals.
We should warn the reader that $\sigma$ is \emph{not} a braiding in the usual sense (for instance, it is not invertible).
The choice of name is justified by the fact that it can be obtained as a limit of the rescaled braiding $q^{(\lambda, \lambda')} \hat{R}_{V(\lambda), V(\lambda')}$ of $\Uqg$-modules, see \cite[Theorem 5.3]{matassa-yuncken}.
As such, it inherits various properties from the braiding $\hat{R}_{V(\lambda), V(\lambda')}$, which are explored in the cited paper.

\begin{remark}
\label{rmk:commutor-cartan}
The category of $\lieg$-crystals is not braided, but rather admits the structure of a coboundary category.
The natural isomorphisms in this case are given by the \emph{crystal commutor} introduced in \cite{crystal-commutor}.
We observe that the Cartan braiding $\sigma_{\calB(\lambda), \calB(\lambda')}$ coincides with the crystal commutor when restricted to the Cartan component of $\calB(\lambda) \otimes \calB(\lambda')$.
One interesting aspect of the former is that it obeys the braid equation (at the cost of not being invertible), while the latter obeys the more complicated axioms of a coboundary category.
\end{remark}

The Cartan braiding does not typically coincide with the flip map, as is typically the case for crystal morphisms of this type.
However this turns to be the case for the subset of extremal elements of the tensor product, as shown in \cite[Lemma 10.1]{matassa-yuncken}.

\begin{lemma}
\label{lem:cartan-extremal}
Let $\lambda, \lambda' \in \dwei$. Then for any $w \in W$ we have
\[
\sigma_{\calB(\lambda), \calB(\lambda')} (b_{w \lambda} \otimes b_{w \lambda'}) = b_{w \lambda'} \otimes b_{w \lambda}.
\]
\end{lemma}

We are going to make frequent use of this property in the following.

\subsection{Right ends}

Before describing the higher-rank graphs we need one more definition, which is that of right end of a crystal element (with respect to a dominant weight).

\begin{definition}
Let $\lambda, \mu \in \dwei$ be such that $\lambda \succeq \mu$.
The \emph{right end} of $\calB(\lambda)$ with respect to $\mu$ is the (set-theoretic) map $\Rend_\mu : \calB(\lambda) \to \calB(\mu)$ defined as follows: writing
\[
\calB(\lambda) \hookrightarrow \calB(\lambda - \mu) \otimes \calB(\mu), \quad
b \mapsto b' \otimes b'',
\]
for the (unique) inclusion of the Cartan component, we set $\Rend_\mu(b) := b''$.
\end{definition}

This definition can be extended to products of connected components as follows. Note that the Cartan component of $\calB = \calB(\lambda_1) \otimes \cdots \otimes \calB(\lambda_n)$ can be identified with $\calB(\lambda)$, where $\lambda = \lambda_1 + \cdots + \lambda_n$.
If $b \in \calB$ is not in the Cartan component we set $\Rend_\mu(b) := 0$, otherwise we define $\Rend_\mu(b)$ as above for $\calB(\lambda)$ using the unique inclusion $\calB(\lambda) \hookrightarrow \calB$.

We adopt the following notational conventions. We write $\Rend_i := \Rend_{\omega_i}$ for the right ends corresponding to the fundamental weights. We write $\rho := \omega_1 + \cdots + \omega_r$ for the smallest regular weight.
Then for any $\lambda \in \dwei$ such that $\lambda \succeq \rho$, and any element $b \in \calB(\lambda)$, we write
\[
\Rend(b) := \left( \Rend_1(b), \ldots, \Rend_r(b) \right).
\]

A convenient way to compute right ends is given in \cite[Proposition 6.7]{matassa-yuncken}, where $\sigma_k$ denotes the action of the Cartan braiding on the factors $k$ and $k + 1$ of a tensor product.

\begin{proposition}
\label{prop:computation-rightend}
The elements $b = b_1 \otimes \cdots \otimes b_n$ belonging to the Cartan component of $\calB(\lambda_1) \otimes \cdots \otimes \calB(\lambda_n)$ can be characterized by the following property:
\[
\sigma_{n - 1} \cdots \sigma_k (b_1 \otimes \cdots \otimes b_n) \neq 0, \quad k = 1, \ldots, n - 1.
\]
Moreover, the right end $\Rend_{\lambda_k} (b)$ is equal to the rightmost tensor factor of this expression.
\end{proposition}

Another useful property concerning right ends is the following \cite[Proposition 6.9]{matassa-yuncken}.

\begin{proposition}
\label{prop:source-identity}
Let $c \in \calB(\rho)$ and write $\Rend(c) = (c_1, \ldots, c_r)$.
Suppose that $c \otimes b$ is in the Cartan component of $\calB(\rho) \otimes \calB(\lambda)$.
Then we have $\Rend_i(c \otimes b) = \Rend_i(c_i \otimes b)$ for $i = 1, \ldots, r$.
\end{proposition}

This identity can be used to compute the source of a path in the definition below.

\subsection{Construction}

Finally we come to the definition of the higher-rank graph given in \cite{matassa-yuncken}.
We note that our conventions here are slightly different (see \cref{rem:conventions-hrg}).

\begin{definition}[The higher-rank graph $\hrg$]
\label{def:higher-rank-g}
Let $\lieg$ be a finite-dimensional complex semisimple Lie algebra of rank $r$.
To it we associate a category $\hrg$ defined as follows.

\begin{itemize}
\item The \emph{vertices} are given by the right ends of $\calB(\rho)$ with $\rho = \omega_1 + \cdots + \omega_r$.
Note that the crystal $\calB(\rho)$ can be identified with the Cartan component of $\calB(\omega_1) \otimes \cdots \otimes \calB(\omega_r)$.
\item The \emph{paths} are given by pairs of the form $(v, b)$, with $v$ a vertex and $b \in \calB(\lambda)$, satisfying the following compatibility condition: writing $v = \Rend(c)$ for some $c \in \calB(\rho)$, we must have that $c \otimes b$ is in the Cartan component of $\calB(\rho) \otimes \calB(\lambda)$.
\item A path $e = (v, b)$ has \emph{range} $r(e) = v$ and \emph{source} $s(e) = \Rend(c \otimes b)$, where as above we write the vertex as $v = \Rend(c)$ for some $c \in \calB(\rho)$.
\item A path $(v, b)$ as above has \emph{degree} $\lambda$, using the natural identification $\dwei \cong \bbN^r$.
\item Given two composable paths $(v, b)$ and $(v', b')$, with $b \in \calB(\lambda)$ and $b' \in \calB(\lambda')$, their \emph{composition} is given by $(v, b) \cdot (v', b') = (v, \phi(b \otimes b'))$, where $\phi: \calB(\lambda) \otimes \calB(\lambda') \to \calB(\lambda + \lambda')$ denotes the projection onto the Cartan component.
\end{itemize}
\end{definition}

Equivalently, the source $s(e)$ of a path $e = (v, b)$ can be computed in terms of \cref{prop:source-identity} as follows: writing $v = \Rend(c) = (c_1, \ldots, c_r)$ for some $c \in \calB(\rho)$, we have
\[
s(e) = \left( \Rend_1(c_1 \otimes b), \ldots, \Rend_r(c_r \otimes b) \right).
\]

The results from the previous subsections can be used to show that $\hrg$ actually defines a category.
Furthermore, we have the following result from \cite[Theorem 7.8]{matassa-yuncken}.

\begin{theorem}
The category $\hrg$ is a higher-rank graph of rank $r$.
\end{theorem}

\begin{remark}
\label{rem:conventions-hrg}
The higher-rank graph $\hrg$ described here is the opposite (as a category) of the one described in \cite[Theorem 7.8]{matassa-yuncken} and denoted there by $\Lambda_{\lieg}$.
The reason for this change in conventions is to have a better match between the directions of paths in $\hrg$ and the conventional way in which the Bruhat graph of $W$ is described.
\end{remark}

The structure of the higher-rank graph $\hrg$ appears to be rather intricate, although various general results are proven in the cited paper.
For instance, we have the following result regarding the vertex set of $\hrg$ from \cite[Proposition 10.3]{matassa-yuncken}.

\begin{proposition}
\label{prop:weyl-vertex}
The Weyl group $W$ embeds into the vertex set of $\hrg$ by $w \mapsto \Rend(b_{w \rho})$.
\end{proposition}

We note that this is essentially a consequence of \cref{lem:cartan-extremal} and the fact that the stabilizer of $\rho$ under the action of $W$ is trivial.
In the following we often make a small abuse of notation and identify the right end $\Rend(b_{w \rho})$ with the element $b_{w \rho}$.
This is justified, since for $b_{w \rho} = b_1 \otimes \cdots \otimes b_r \in \calB(\rho) \subseteq \calB(\omega_1) \otimes \cdots \otimes \calB(\omega_r)$ we get $\Rend(b_{w \rho}) = (b_1, \ldots, b_r)$.

\begin{remark}
In general the vertex set of $\hrg$ does not correspond to the Weyl group of $\lieg$.
This is shown for the Lie algebra $\lieg = C_2$ in \cite[Example 10.6]{matassa-yuncken}.
\end{remark}

We can define a partial order on the vertex set of $\hrg$ as follows. First, any crystal $\calB(\lambda)$ is partially ordered by setting $b \leq b'$ whenever $b = \tilde{F}_{i_1} \cdots \tilde{F}_{i_n} b'$ for some Kashiwara operators.
Then, for vertices $v = (b_1, \ldots, b_r)$ and $v' = (b_1', \ldots, b_r')$, we say that $v \leq v'$ if $b_i \leq b_i'$ for every $i$.
It is proven in \cite[Proposition 7.11]{matassa-yuncken} that for any path $e \in \hrg$ we have $r(e) \leq s(e)$ (once again, recall the opposite conventions from \cref{rem:conventions-hrg}).
Considering $\hrg$ as a directed graph, this shows that all cycles must have length one (that is, they are loops).

\section{Some technical results}
\label{sec:technical-results}

In this section we prove some results needed for the next one, where we discuss the problem of embedding the Bruhat graph into the higher-rank graph $\hrg$.
We begin by exploring the interplay between the Bruhat order and the Cartan component of a tensor product of connected crystals.
The main result here is that $b_{w \lambda} \otimes b_{w' \lambda'}$ is in the Cartan component of $\calB(\lambda) \otimes \calB(\lambda')$, provided $w \geq w'$ with respect to the Bruhat order.
This is used to study certain paths in $\hrg$ corresponding to any such pair $(w, w')$ of Weyl group elements.

\subsection{Bruhat order and crystals}

Recall the following standard fact related to the length function $\ell$ on a Weyl group: we have $\ell(s_i w) > \ell(w)$ if and only if $w^{-1} \alpha_i \succ 0$.
This has the following well-known consequence for the maps $\varphi_i$ and $\varepsilon_i$ of a crystal.

\begin{lemma}
\label{lem:epsilon-w}
Let $\lambda \in \dwei$ and $w \in W$. The following holds:
\begin{enumerate}
\item if $\ell(s_i w) > \ell(w)$ then $\varepsilon_i(b_{w \lambda}) = 0$,
\item if $\ell(s_i w) < \ell(w)$ then $\varphi_i(b_{w \lambda}) = 0$.
\end{enumerate}
\end{lemma}

\begin{proof}
(1) Suppose on the contrary that $\varepsilon_i(b_{w \lambda}) > 0$, so that the element $\tilde{E}_i b_{w \lambda}$ is non-zero and of weight $w \lambda + \alpha_i$. Then, acting with $w^{-1}$ on it, we obtain another element of $\calB(\lambda)$ of weight $\lambda + w^{-1} \alpha_i$.
Since $\ell(s_i w) > \ell(w)$ implies $w^{-1} \alpha_i \succ 0$, we find that $\lambda + w^{-1} \alpha_i \succ \lambda$.
But this is a contradiction, since $\lambda$ is the highest weight of $\calB(\lambda)$.

(2) Similar to (1), by considering $\tilde{F}_i b_{w \lambda}$ and using $w^{-1} \alpha_i \prec 0$.
\end{proof}

We use this to investigate the action of $\tilde{F}_i$ on particular elements of a tensor product.

\begin{lemma}
\label{lem:action-F}
Let $\lambda, \lambda' \in \dwei$ and $w, w' \in W$.
Let $n = (w \lambda, \alpha_i^\vee)$ and $n' = (w' \lambda', \alpha_i^\vee)$.
Suppose that $\ell(s_i w) > \ell(w)$ and $\ell(s_i w') > \ell(w')$. Then:
\begin{enumerate}
\item for $k \leq n$ we have $\tilde{F}_i^k (b_{w \lambda} \otimes b_{w' \lambda'}) = \tilde{F}_i^k b_{w \lambda} \otimes b_{w' \lambda'}$,
\item for $n < k \leq n + n'$ we have $\tilde{F}_i^k (b_{w \lambda} \otimes b_{w' \lambda'}) = \tilde{F}_i^n b_{w \lambda} \otimes \tilde{F}_i^{k - n} b_{w' \lambda'}$.
\end{enumerate}
\end{lemma}

\begin{proof}
First we observe that both $n$ and $n'$ are non-negative, so that the claims make sense. To see this, we rewrite $n = (\lambda, (w^{-1} \alpha_i)^\vee)$ using $W$-invariance (similarly for $n'$). The condition $\ell(s_i w) > \ell(w)$ implies $w^{-1} \alpha_i \succ 0$, which gives $n \geq 0$ since $\lambda$ is dominant.

(1) We have $\varepsilon_i(b_{w \lambda}) = 0$ by \cref{lem:epsilon-w}, since $\ell(s_i w) > \ell(w)$.
This means that $n = \varphi_i(b_{w \lambda})$.
Furthermore, this gives $\varphi_i(\tilde{F}_i^k b_{w \lambda}) = n - k$ for $k \leq n$.
Similarly we have $\varepsilon_i(b_{w' \lambda'}) = 0$ and $n' = \varphi_i(b_{w' \lambda'})$.
Then the claim easily follows from the tensor product rule \eqref{eq:tensor-rule}, since we have $\varphi_i(\tilde{F}_i^k b_{w \lambda}) > \varepsilon_i(b_{w' \lambda'})$ for $k < n$, which means that we always act on the first factor.

(2) Since $\varphi_i(\tilde{F}_i^n b_{w \lambda}) = 0$, all successive applications of $\tilde{F}_i$ are on the second factor.
\end{proof}

The previous result can be rephrased in terms of the Weyl group action on crystals.

\begin{corollary}
\label{cor:action-F-weyl}
Under the assumptions and notations \cref{lem:action-F}, we have
\[
\tilde{F}_i^n (b_{w \lambda} \otimes b_{w' \lambda'}) = \wgc_i b_{w \lambda} \otimes b_{w' \lambda'}, \quad
\wgc_i (b_{w \lambda} \otimes b_{w' \lambda'}) = \wgc_i b_{w \lambda} \otimes \wgc_i b_{w' \lambda'}.
\]
\end{corollary}

\begin{proof}
By definition of the Weyl group action we have $\wgc_i b_{w \lambda} = \tilde{F}_i^n b_{w \lambda}$ and $\wgc_i b_{w' \lambda'} = \tilde{F}_i^{n'} b_{w' \lambda'}$, since $n = (w \lambda, \alpha_i^\vee) \geq 0$ and $n' = (w' \lambda', \alpha_i^\vee) \geq 0$. Similarly we have $\wgc_i (b_{w \lambda} \otimes b_{w' \lambda'}) = \tilde{F}_i^{n + n'} (b_{w \lambda} \otimes b_{w' \lambda'})$. Then the identities immediately follow from \cref{lem:action-F}.
\end{proof}

Before proceeding, we need the following well-known result concerning the Bruhat order.

\begin{lemma}
\label{lem:reduced-removal}
Let $w, w' \in W$ be such that $w > w'$ in the Bruhat order. Let $n = \ell(w)$ and $m = \ell(w) - \ell(w')$.
Fix a reduced expression $w = s_{i_1} \cdots s_{i_n}$.
Then there exists a decreasing sequence $j_1 > \cdots > j_m$ of indices (for the decomposition of $w$) such that:
\begin{itemize}
\item the word $s_{i_1} \cdots \hat{s}_{j_k} \cdots \hat{s}_{j_1} \cdots s_{i_n}$, obtained from $w$ by removing the letters $\hat{s}_{j_k}, \ldots \hat{s}_{j_1}$, is a reduced expression for any $k = 1, \ldots, m$,
\item we have $s_{i_1} \cdots \hat{s}_{j_m} \cdots \hat{s}_{j_1} \cdots s_{i_n} = w'$.
\end{itemize}
\end{lemma}

\begin{proof}
This is proven in \cite[Section 2.7]{bjorner-brenti} for the increasing case, that is by removing letters from the left.
The decreasing case follows by applying this result to $w^{-1} > (w')^{-1}$, since this is equivalent to $w > w'$ and inversion reverses the reduced decomposition.
\end{proof}

We can use the previous results to prove the following proposition, which relates the Bruhat order on $W$ to the Cartan component of a tensor product.

\begin{proposition}
\label{prop:cartan-bruhat}
Let $\lambda, \lambda' \in \dwei$ and $w, w' \in W$.
Suppose that $w \geq w'$ in the Bruhat order.
Then $b_{w \lambda} \otimes b_{w' \lambda'}$ is in the Cartan component of $\calB(\lambda) \otimes \calB(\lambda')$.
\end{proposition}

\begin{proof}
Fix a reduced decomposition $w = s_{i_1} \cdots s_{i_n}$.
If $w = w'$ we can apply repeatedly \cref{cor:action-F-weyl} to get $\wgc_1 \cdots \wgc_n (b_\lambda \otimes b_{\lambda'})
= b_{w \lambda} \otimes b_{w \lambda'}$.
This shows that $b_{w \lambda} \otimes b_{w \lambda'}$ is in the Cartan component, since it is obtained from $b_\lambda \otimes b_{\lambda'}$ by the action of Kashiwara operators.

If $w \neq w'$ we proceed as follows.
Write $w_k = s_{i_k} \cdots s_{i_n}$ for $k = 1, \ldots, n$ and note that this is reduced for each $k$.
By \cref{lem:reduced-removal}, we can obtain $w'$ from $w$ by removing a decreasing sequence $j_1 > \cdots > j_m$, such that at each step we get a reduced decomposition.
Let us write $w'_k = \tilde{s}_k \cdots \tilde{s}_n$, where $\tilde{s}_k = 1$ if $k = j_a$ for some $a$ and $\tilde{s}_k = s_{i_k}$ otherwise.
Note that $w'_1 = w'$, since this is exactly obtained by removing the letters $s_{j_1}, \ldots, s_{j_m}$.

Define the map $T_k$ for $k = 1, \ldots, n$ as follows: if $\tilde{s}_k = s_{i_k}$ we set $T_k = \wgc_{i_k}$, while if $\tilde{s}_k = 1$ we set $T_k = \tilde{F}^{a_k}_{i_k}$ where $a_k = (w_{k + 1} \lambda, \alpha_{i_k}^\vee)$ (the discussion below shows that these numbers are non-negative).
We claim that $T_1 \cdots T_n (b_\lambda \otimes b_{\lambda'}) = b_{w \lambda} \otimes b_{w' \lambda'}$, which would prove that $b_{w \lambda} \otimes b_{w' \lambda'}$ belongs to the Cartan component, since each $T_k$ consists of Kashiwara operators.

The base case $T_n (b_\lambda \otimes b_{\lambda'}) = b_{w_n \lambda} \otimes b_{w'_n \lambda'}$ follows immediately from \cref{cor:action-F-weyl}.
Now suppose that $T_{k + 1} \cdots T_n (b_\lambda \otimes b_{\lambda'}) = b_{w_{k + 1} \lambda} \otimes b_{w'_{k + 1} \lambda'}$ for some $k + 1$.
Observe that $\ell(s_{i_k} w_{k + 1}) > \ell(w_{k + 1})$ and $\ell(s_{i_k} w'_{k + 1}) > \ell(w'_{k + 1})$, where the first follows from the fact that we chose a reduced decomposition for $w$, while the second follows from \cref{lem:reduced-removal}.
Note that the first identity also shows that the numbers $a_k = (w_{k + 1} \lambda, \alpha_{i_k}^\vee)$ are non-negative.
Then applying \cref{cor:action-F-weyl} once again leads to $T_k \cdots T_n (b_\lambda \otimes b_{\lambda'}) = b_{w_k \lambda} \otimes b_{w'_k \lambda'}$, which proves the result.
\end{proof}

\subsection{Consequences for paths}

We now investigate the consequences of \cref{prop:cartan-bruhat} for paths in the higher-rank graph $\hrg$, beginning with the following general result.

\begin{proposition}
\label{prop:paths-bruhat}
Suppose $w \geq w'$ in the Bruhat order. Then, for any $\lambda \in \dwei$, we have that the pair $(b_{w \rho}, b_{w' \lambda})$ is a path in $\hrg$ of color $\lambda$ and range $b_{w \rho}$.
\end{proposition}

\begin{proof}
By \cref{def:higher-rank-g}, the pair $(v, b)$ with $b \in \calB(\lambda)$ is a path in $\hrg$ if, choosing $c \in \calB(\rho)$ such that $v = \Rend(c)$, we have that $c \otimes b$ is in the Cartan component of $\calB(\rho) \otimes \calB(\lambda)$.
Consider $(b_{w \rho}, b_{w' \lambda})$, where we identify $b_{w \rho}$ with its right end as usual.
Since $w \geq w'$, it follows from \cref{prop:cartan-bruhat} that $b_{w \rho} \otimes b_{w' \lambda}$ is in the Cartan component of $\calB(\rho) \otimes \calB(\lambda)$, which shows that it is a path. It is immediate from the definitions that is it of color $\lambda$ and range $b_{w \rho}$.
\end{proof}

Determining the sources of such paths is a more complicated matter.
To proceed, we introduce the following notions (which are more or less standard).

\begin{definition}
Let $\gamma$ be a positive root and write $\gamma = \sum_{i \in I} c_i \alpha_i$ in terms of the simple roots.
Then its \emph{root support} is defined as the set $\suppr(\gamma) := \{ i \in I : c_i \neq 0 \}$.

Similarly, let $\lambda$ be a dominant weight and write $\lambda = \sum_{i \in I} c_i' \omega_i$ in terms of the fundamental weights.
Then its \emph{weight support} is defined as the set $\suppw(\lambda) := \{ i \in I : c_i' \neq 0 \}$.
\end{definition}

\begin{remark}
The two distinct notations are used to avoid ambiguity in the case of positive roots which are also dominant weights (although the use should be clear from the context).
\end{remark}

It follows immediately from the definition that $(\omega_i, \gamma) > 0$ for $i \in \suppr(\gamma)$ and is zero otherwise.
Similarly $(\lambda, \alpha_i) > 0$ for $i \in \suppw(\lambda)$ and is zero otherwise.

These two notions of support can be used to describe two cases in which we can compute the right ends easily, as illustrated in the next result.

\begin{lemma}
\label{lem:source-end}
Let $w \in W$ and $t \in T$. Suppose that $\ell(w t) > \ell(w)$.
Denote by $\gamma$ be the positive root corresponding to $t$.
Then for any $\lambda \in \dwei$ we have
\[
\Rend_i (b_{w t \omega_i} \otimes b_{w \lambda}) = b_{w \omega_i},
\]
provided that either $i \notin \suppr(\gamma)$ or $i \in \suppw(\lambda)$.
\end{lemma}

\begin{proof}
Recall that the reflection $t = t_\gamma$ acts on a weight $\mu$ by $t \mu = \mu - (\mu, \gamma^\vee) \gamma$.
Hence we get $t \omega_i = \omega_i$ if $i \notin \suppr(\gamma)$.
Also recall that the Cartan braiding coincides with the flip map for elements in the orbit of the Weyl group by \cref{lem:cartan-extremal}.
Using these facts we get
\[
\Rend_i (b_{w t \omega_i} \otimes b_{w \lambda})
= \Rend_i (b_{w \omega_i} \otimes b_{w \lambda}) = b_{w \omega_i}.
\]

Now consider the case $i \in \suppw(\lambda)$.
Then we have a unique morphism of crystals
\[
\calB(\lambda) \hookrightarrow \calB(\omega_i) \otimes \calB(\lambda - \omega_i),
\]
with $\calB(\lambda)$ being the Cartan component of the tensor product.
Under this identification, the extremal element $b_{w \lambda}$ is mapped to $b_{w \omega_i} \otimes b_{w (\lambda - \omega_i)}$.
According to \cref{prop:computation-rightend}, the right end $\Rend_i(b_{w t \omega_i} \otimes b_{w \lambda})$ can be obtained as the rightmost factor of the expression
\[
\sigma_2 \sigma_1 (b_{w t \omega_i} \otimes b_{w \omega_i} \otimes b_{w (\lambda - \omega_i)}).
\]
The assumption $\ell(w t) > \ell(w)$ guarantees that $b_{w t \omega_i} \otimes b_{w \omega_i}$ is in the Cartan component of $\calB(\omega_i) \otimes \calB(\omega_i)$, by \cref{prop:cartan-bruhat}.
We also recall that the Cartan braiding is the identity in this case (the identity clearly being a morphism of crystals here).
Then we compute
\[
\sigma_2 \sigma_1 (b_{w t \omega_i} \otimes b_{w \omega_i} \otimes b_{w (\lambda - \omega_i)})
= \sigma_2 (b_{w t \omega_i} \otimes b_{w \omega_i} \otimes b_{w (\lambda - \omega_i)})
= b_{w t \omega_i} \otimes b_{w (\lambda - \omega_i)} \otimes b_{w \omega_i}.
\]
We note that the last step follows from the fact that both elements are in the orbit of $w$, using \cref{lem:cartan-extremal}.
From this we conclude that $\Rend_i (b_{w t \omega_i} \otimes b_{w \lambda}) = b_{w \omega_i}$.
\end{proof}

The lemma above allows us to easily compute the sources of certain paths.

\begin{corollary}
\label{cor:source-edge}
With notation as in \cref{lem:source-end}, suppose that $\suppr(\gamma) \subseteq \suppw(\lambda)$. Then we have that $(b_{w t \rho}, b_{w \lambda})$ is a path with source $b_{w \rho}$ and range $b_{w t \rho}$.
\end{corollary}

\begin{proof}
It is clear that its range is given by $b_{w t \rho}$.
To determine its source, we need to compute
\[
\Rend(b_{w t \rho} \otimes b_{w \lambda}) = \left( \Rend_1(b_{w t \omega_1} \otimes b_{w \lambda}), \ldots, \Rend_r(b_{w t \omega_r} \otimes b_{w \lambda}) \right).
\]
By \cref{lem:source-end} we know that $\Rend_i (b_{w t \omega_i} \otimes b_{w \lambda}) = b_{w \omega_i}$ if $i \notin \suppr(\gamma)$ or $i \in \suppw(\lambda)$.
Then the assumption $\suppr(\gamma) \subseteq \suppw(\lambda)$ guarantees that this is the case for all $i \in I$. Hence
\[
\left( \Rend_1(b_{w t \omega_1} \otimes b_{w \lambda}), \ldots, \Rend_r(b_{w t \omega_r} \otimes b_{w \lambda}) \right)
= (b_{w \omega_1}, \ldots, b_{w \omega_r}) \equiv b_{w \rho}. \qedhere
\]
\end{proof}

\section{Embedding the Bruhat graph}
\label{sec:embedding}

We are now going to use the results of the previous section to discuss in what sense the Bruhat graph can be embedded in the higher-rank graph $\hrg$.
We treat separately the cases of the weak Bruhat graphs and the (strong) Bruhat graph (as well as discussing how we can regard them as colored graphs).
We show that the right weak Bruhat graph embeds naturally in the skeleton of $\hrg$, unlike the left one.
For the Bruhat graph we necessarily have to look for embeddings into the whole $\hrg$, not just its skeleton, and we show that many such embeddings are possible.
Finally we illustrate and discuss these results in the case of $\lieg = A_2$.

\subsection{The weak Bruhat graphs}

Let us begin by discussing the case of the right weak Bruhat graph $\bruhatR$, whose edges are of the form $w \to w s_i$ with $\ell(w s_i) > \ell(w)$.

We can make $\bruhatR$ into a colored graph with colors $I = \{ 1, \ldots, r \}$ in a natural way, by assigning the color $i$ to each edge of the form $w \to w s_i$ (this should be seen as corresponding to the right multiplication in the Weyl group by the simple reflection $s_i$).
Similarly, we can regard the skeleton of the higher-rank graph $\hrg$ as a colored graph with colors $\{ \omega_1, \ldots, \omega_r \}$.
In the following we identify these sets of colors by $k \equiv \omega_k$.

As previously discussed in \cref{prop:weyl-vertex}, we have an inclusion
\[
\iota_W : W \to \hrg, \quad
w \mapsto b_{w \rho}
\]
of the Weyl group $W$ into the vertex set of $\hrg$.
Here as usual we make no distinction between the crystal element $b_{w \rho} \in \calB(\rho)$ and its right end in $\hrg$.

Hence we look for embeddings of the colored graph $\bruhatR$ into the skeleton of $\hrg$ which are compatible with $\iota_W$, meaning that the vertices are embedded using $\iota_W$.

\begin{theorem}
\label{thm:embedding-rightweak}
There is a unique embedding (of colored graphs) of the right weak Bruhat graph $\bruhatR$ into the skeleton of $\hrg$ which is compatible with the inclusion $\iota_W$.

It sends the edge $w \to w s_i$ of $\bruhatR$ to the edge $(b_{w s_i \rho}, b_{w \omega_i})$ of $\hrg$.
\end{theorem}

\begin{proof}
First we show the existence of such an embedding.
For the vertices this is given by $\iota_W$, that is $w \mapsto b_{w \rho}$.
Next, we claim that $(b_{w s_i \rho}, b_{w \omega_i})$ is an edge in $\hrg$ of color $\omega_i$, source $b_{w \rho}$ and range $b_{w s_i \rho}$. This follows from \cref{cor:source-edge} since, using the notation employed there, we have $\gamma = \alpha_i$, $\lambda = \omega_i$ and $\suppr(\alpha_i) = \suppw(\omega_i) = \{ i \}$.
This shows the existence of the required embedding of colored graphs, so all that is left is to prove uniqueness.

Clearly an edge $w \to w s_i$ in $\bruhatR$ must be mapped to an edge of color $\omega_i$, source $b_{w \rho}$ and range $b_{w s_i \rho}$ in $\hrg$. Then such an edge must be of the form $(b_{w s_i \rho}, b)$ with $b \in \calB(\omega_i)$, since the vertex in this pair must be equal to the range, and the crystal of $b$ determines the color.
The source of this edge can be computed by determining the right ends $\Rend_k(b_{w s_i \omega_k} \otimes b)$ for $k = 1, \ldots, r$.
Consider $k = i$, so that $b_{w s_i \omega_i} \otimes b$ must be in the Cartan component of $\calB(\omega_i) \otimes \calB(\omega_i)$. In this case the Cartan braiding is the identity, from which we see that $\Rend_i(b_{w s_i \omega_i} \otimes b) = b$.
But, in order for the source to be $b_{w \rho}$, we must have $\Rend_i(b_{w s_i \omega_i} \otimes b) = b_{w \omega_i}$.
This forces $b = b_{w \omega_i}$ and hence uniqueness of the edge in the skeleton of $\hrg$.
\end{proof}

\begin{figure}[h]
\centering

\begin{tikzpicture}[
scale = 1.5,
vertex/.style = {align=center, inner sep=2pt},
Rarr/.style = {->, red},
Barr/.style = {->, blue, dashed},
shadow/.style = {white, line width=3pt},
Rloop/.style = {->, red, out=165, in=195, loop},
Bloop/.style = {->, blue, out=15, in=-15, loop, dashed}
]
\node (v1) at ( 0, 0) [vertex] {$1$};
\node (v2) at (-2,-1) [vertex] {$s_2$};
\node (v3) at ( 2,-1) [vertex] {$s_1$};
\node (v4) at (-2,-2) [vertex] {$s_1 s_2$};
\node (v5) at ( 2,-2) [vertex] {$s_2 s_1$};
\node (v6) at ( 0,-3) [vertex] {$s_1 s_2 s_1 = s_2 s_1 s_2$};

\draw [Barr] (v1)--(v2);
\draw [Rarr] (v1)--(v3);
\draw [Rarr] (v2)--(v4);
\draw [Barr] (v3)--(v5);
\draw [Barr] (v4)--(v6);
\draw [Rarr] (v5)--(v6);
\end{tikzpicture}
\qquad
\begin{tikzpicture}[
scale = 1.5,
vertex/.style = {align=center, inner sep=2pt},
Rarr/.style = {->, red},
Barr/.style = {->, blue, dashed},
shadow/.style = {white, line width=3pt},
Rloop/.style = {->, red, out=165, in=195, loop},
Bloop/.style = {->, blue, out=15, in=-15, loop, dashed}
]
\node (v1) at ( 0, 0) [vertex] {$1$};
\node (v2) at (-2,-1) [vertex] {$s_2$};
\node (v3) at ( 2,-1) [vertex] {$s_1$};
\node (v4) at (-2,-2) [vertex] {$s_1 s_2$};
\node (v5) at ( 2,-2) [vertex] {$s_2 s_1$};
\node (v6) at ( 0,-3) [vertex] {$s_1 s_2 s_1 = s_2 s_1 s_2$};

\draw [Barr] (v1)--(v2);
\draw [Rarr] (v1)--(v3);
\draw [Rarr] (v2)--(v5);
\draw [shadow] (v3)--(v4);
\draw [Barr] (v3)--(v4);
\draw [Rarr] (v4)--(v6);
\draw [Barr] (v5)--(v6);
\end{tikzpicture}
    
\caption{The left and right weak Bruhat graphs for the Weyl group $S_3$, considered as colored graphs. Here the red solid edges correspond to the color $1$, while the blue dashed edges correspond to the color $2$.}
\label{fig:weak-bruhat-S3}
\end{figure}

The situation turns out to be fairly different for the left weak Bruhat graph.

\begin{proposition}
\label{prop:embedding-leftweak}
In general there is no embedding (of colored graphs) of the left weak Bruhat graph $\bruhatL$ into the skeleton of $\hrg$ which is compatible with the inclusion $\iota_W$.
\end{proposition}

\begin{proof}
Consider the left weak Bruhat graph of $W = S_3$, which is depicted in \cref{fig:weak-bruhat-S3}.
We have an edge of color $2$ from $s_1 s_2$ to $s_1 s_2 s_1 = s_2 s_1 s_2$.
On the other hand, the skeleton of $\hrg$ depicted in \cref{fig:skeleton-graph-sl3} shows that there is no edge of color $\omega_2$ between those two vertices.
\end{proof}

\begin{remark}
The asymmetry between the left and right weak Bruhat orders can be related to the fact that the vertices of $\hrg$ are defined using right ends (instead of left ends).
\end{remark}

In the next subsection we are going to consider the question of embedding the (strong) Bruhat graph $\bruhat$ into the \emph{whole} $\hrg$, as opposed to just its skeleton.
Since $\bruhatL$ is a subgraph of $\bruhat$, the results will have implications for embeddings of $\bruhatL$ as well.

\subsection{The Bruhat graph}

Before we begin considering the question of embedding the (strong) Bruhat graph $\bruhat$ into $\hrg$, we must discuss how to make the former into a colored graph.
Here there is no canonical way of doing so, unlike in the weak case.
For this reason we are going to consider some general choices of coloring, described as follows.

\begin{definition}
\label{def:compatible-coloring}
Let $\coloring$ be a function from the edge set of $\bruhat$ to the dominant weights $\dwei$.
We say that $\coloring$ is a \emph{compatible coloring} if the following holds: given an edge $e: w  \to w t$, with $t$ corresponding to the positive root $\gamma$, we have that $\suppr(\gamma) \subseteq \suppw(\coloring(e))$.
\end{definition}

More concretely, if $\gamma = \sum_{i \in I} c_i \alpha_i$ and $\coloring(e) = \sum_{i \in I} c_i' \omega_i$, we ask that $c_i \neq 0$ implies $c_i' \neq 0$.
Any such choice of compatible coloring gives an embedding, as we now show.

\begin{theorem}
\label{thm:embedding-bruhat}
Consider the Bruhat graph $\bruhat$ as a colored graph via a compatible coloring $\coloring$.
Then there exists an embedding of $\bruhat$ into $\hrg$ given as follows:
\begin{itemize}
\item each vertex $w \in W$ is sent to the (right end of the) crystal element $b_{w \rho}$, 
\item each edge $e: w \to w t$ is sent to the path $(b_{w t \rho}, b_{w \lambda})$ with $\lambda = \coloring(e)$.
\end{itemize}
\end{theorem}

\begin{proof}
Let $\gamma$ be the positive root corresponding to the reflection $t$.
By \cref{lem:source-end} we have that $\Rend_i (b_{w t \omega_i} \otimes b_{w \mu}) = b_{w \omega_i}$ if $i \notin \suppr(\gamma)$ or $i \in \suppw(\mu)$, for any dominant weight $\mu$.
The condition in the definition of $\coloring$ guarantees that $\suppr(\gamma) \subseteq \suppw(\coloring(e))$.
Then we have $\suppr(\gamma) \subseteq \suppw(\lambda)$, which guarantees that $\Rend_i (b_{w t \omega_i} \otimes b_{w \lambda}) = b_{w \omega_i}$ for any $i \in I$.
From this it is immediate to see that $(b_{w t \rho}, b_{w \lambda})$ is a path from $b_{w \rho}$ to $b_{w t \rho}$ of color $\coloring(e)$, as required.
\end{proof}

Further conditions can potentially be imposed on this embedding to select a "natural" candidate, such as requiring some form of compatibility with the composition of edges in the skeleton of $\hrg$ corresponding to $\bruhatR$.
We leave such questions for future research.

\subsection{Example of $A_2$}
\label{sec:example-a2}

We illustrate some of the features of the previous results in the example of the simple Lie algebra $\lieg = A_2 = \mathfrak{sl}_3$ of rank $2$. In this case the Weyl group is $W = S_3$ and the corresponding Bruhat graph is depicted in \cref{fig:bruhat-s3}.

\begin{figure}[h]
\centering

\begin{tikzpicture}[
scale = 1.5,
vertex/.style = {align=center, inner sep=2pt},
arr/.style = {->, black},
shadow/.style = {white, line width=3pt}
]
\node (v1) at ( 0, 0) [vertex] {$1$};
\node (v2) at (-2,-1) [vertex] {$s_2$};
\node (v3) at ( 2,-1) [vertex] {$s_1$};
\node (v4) at (-2,-2) [vertex] {$s_1 s_2$};
\node (v5) at ( 2,-2) [vertex] {$s_2 s_1$};
\node (v6) at ( 0,-3) [vertex] {$s_1 s_2 s_1 = s_2 s_1 s_2$};

\draw [arr] (v1)--(v2);
\draw [arr] (v1)--(v3);
\draw [shadow] (v1)--(v6);
\draw [arr] (v1)--(v6);

\draw [arr] (v2)--(v4);
\draw [arr] (v2)--(v5);

\draw [arr] (v3)--(v4);
\draw [arr] (v3)--(v5);

\draw [arr] (v4)--(v6);

\draw [arr] (v5)--(v6);
\end{tikzpicture}
\caption{The (strong) Bruhat graph $\bruhat$ for $W = S_3$.}
\label{fig:bruhat-s3}
\end{figure}

There are two fundamental crystals in this case, for which we use the notations
\[
\calB(\omega_1) : a_1 \xrightarrow{1} a_2 \xrightarrow{2} a_3, \quad
\calB(\omega_2) : b_1 \xrightarrow{2} b_2 \xrightarrow{1} b_3.
\]
The vertex set of $\hrg$ consist of $6$ elements, which can be identified with the elements of the Weyl group $S_3$ (see also the next section).
The skeleton of $\hrg$ is depicted in \cref{fig:skeleton-graph-sl3}, but we remark that we have removed the loops from each vertex for readability.
This skeleton is derived in \cite[Example 7.9]{matassa-yuncken} (but recall the opposite conventions as in \cref{rem:conventions-hrg}).

\begin{figure}[h]
\centering
    
\begin{tikzpicture}[
scale = 2,
vertex/.style = {align=center, inner sep=2pt},
Rarr/.style = {->, red},
Barr/.style = {->, blue, dashed},
shadow/.style = {white, line width=3pt},
Rloop/.style = {->, red, out=165, in=195, loop},
Bloop/.style = {->, blue, out=15, in=-15, loop, dashed}
]
\node (v1) at ( 0, 0) [vertex] {$(a_1, b_1) \equiv 1$};
\node (v2) at (-2,-1) [vertex] {$(a_1, b_2) \equiv s_2$};
\node (v3) at ( 2,-1) [vertex] {$(a_2, b_1) \equiv s_1$};
\node (v4) at (-2,-2) [vertex] {$(a_2, b_3) \equiv s_1 s_2$};
\node (v5) at ( 2,-2) [vertex] {$(a_3, b_2) \equiv s_2 s_1$};
\node (v6) at ( 0,-3) [vertex] {$(a_3, b_3) \equiv s_1 s_2 s_1 = s_2 s_1 s_2$};

\draw [Barr] (v1)--(v2);
\draw [Rarr] (v1)--(v3);

\draw [Rarr, transform canvas={xshift=-0.3em}] (v2)--(v4);
\draw [Barr, transform canvas={xshift=0.3em}] (v2)--(v4);
\draw [Rarr] (v2)--(v5);
\draw [Rarr] (v2)--(v6);

\draw [shadow] (v3)--(v4);
\draw [Barr] (v3)--(v4);
\draw [shadow] (v3)--(v6);
\draw [Barr] (v3)--(v6);
\draw [Rarr, transform canvas={xshift=-0.3em}] (v3)--(v5);
\draw [Barr, transform canvas={xshift=0.3em}] (v3)--(v5);

\draw [Rarr] (v4)--(v6);
\draw [Barr] (v5)--(v6);
\end{tikzpicture}
    
\caption{The skeleton of the $2$-graph $\hrg$ for $\lieg = A_2$ (with the loops on each vertex removed for readability).
The red solid edges correspond to the color $\omega_1$, while the blue dashed edges correspond to the color $\omega_2$.}
\label{fig:skeleton-graph-sl3}
\end{figure}

First we compare the skeleton to the weak Bruhat graphs appearing in \cref{fig:weak-bruhat-S3}.
We clearly see that the right weak Bruhat graph $\bruhatR$ embeds into the skeleton of $\hrg$, as follows from the general \cref{thm:embedding-rightweak}.
On the other hand, there is no such embedding for the left weak Bruhat graph $\bruhatL$, as also observed in \cref{prop:embedding-leftweak}.
For instance, there is no edge of color $\omega_2$ in \cref{fig:skeleton-graph-sl3} corresponding to the edge of color $2$ from $s_1 s_2$ to $s_2 s_1 s_2$ in \cref{fig:weak-bruhat-S3}.

Some of the remaining edges in \cref{fig:weak-bruhat-S3} can be interpreted in terms of embeddings of the Bruhat graph $\bruhat$, as in \cref{thm:embedding-bruhat}.
However, note that not all edges of this graph appear in the skeleton of $\hrg$, since the latter only contains paths of colors $\omega_1$ and $\omega_2$. For instance, the edge from $1$ to $s_1 s_2 s_1$ of $\bruhat$ can be embedded into a path of color $\omega_1 + \omega_2$.

Even after all these identifications, there remain edges in the skeleton of $\hrg$ that do not correspond to any edges in $\bruhat$ (such as the edge $s_2 \to s_1 s_2 s_1$ of color $\omega_1$).
This can be seen as a consequence of the \emph{factorization property} which defines a higher-rank graph: for instance, considering the path $1 \to s_1 s_2 s_1$ of color $\omega_1 + \omega_2$ mentioned above, it should be possible to (uniquely) decompose it in terms of edges of colors $\omega_1$ and $\omega_2$.
This would not be possible if we did not have some extra edges in the skeleton (such as $s_2 \to s_1 s_2 s_1$).

In this particular example, it is possible to show that all the edges in the skeleton of $\hrg$ can be seen as arising from the result stated in \cref{prop:paths-bruhat}, as we do below.

\begin{remark}
We can check that all the edges in the skeleton of $\hrg$ for $\lieg = A_2$ correspond to $(b_{w \rho}, b_{w' \omega_i})$ with $w \geq w'$, as in \cref{prop:paths-bruhat}.
For $\calB(\omega_1)$ we have the identifications
\[
a_1 = b_{\omega_1}, \quad
a_2 = b_{s_1 \omega_1}, \quad
a_3 = b_{s_2 s_1 \omega_1}.
\]
Then we can easily list all edges $(b_{w \rho}, b_{w' \omega_1})$ with $w \geq w'$. Clearly any such edge has range $b_{w \rho}$, while the corresponding sources can be computed to be
\[
\begin{array}{lll}
s(b_{\rho}, b_{\omega_1}) = b_{\rho}, &
s(b_{s_1 \rho}, b_{\omega_1}) = b_{\rho}, &
s(b_{s_1 \rho}, b_{s_1 \omega_1}) = b_{s_1 \rho}, \\
s(b_{s_2 \rho}, b_{\omega_1}) = b_{s_2 \rho}, &
s(b_{s_1 s_2 \rho}, b_{\omega_1}) = b_{s_2 \rho}, &
s(b_{s_1 s_2 \rho}, b_{s_1 \omega_1}) = b_{s_1 s_2\rho}, \\
s(b_{s_2 s_1 \rho}, b_{\omega_1}) = b_{s_2 \rho}, &
s(b_{s_1 s_2 \rho}, b_{s_1 \omega_1}) = b_{s_1 \rho}, &
s(b_{s_1 s_2 \rho}, b_{s_2 s_1 \omega_1}) = b_{s_2 s_1 \rho}, \\
s(b_{s_1 s_2 s_1 \rho}, b_{\omega_1}) = b_{s_2 \rho}, &
s(b_{s_1 s_2 s_1 \rho}, b_{s_1 \omega_1}) = b_{s_1 s_2 \rho}, &
s(b_{s_1 s_2 s_1 \rho}, b_{s_2 s_1 \omega_1}) = b_{s_1 s_2 s_2 \rho}.
\end{array}
\]
Inspection of \cref{fig:skeleton-graph-sl3} shows that these account for all the red edges (similar computations can be performed starting from $\calB(\omega_2)$ to obtain all the blue edges).

We remark that we should not expect a similar phenomenon for general $\lieg$.
What is special about $\lieg = A_2$ (or more generally $A_r$) is that the fundamental crystals $\calB(\omega_1)$ and $\calB(\omega_2)$ are both \emph{minuscule}, meaning that all their elements are in the $W$-orbit of the highest weight.
\end{remark}

\section{Right ends and keys}
\label{sec:right-ends}

In this section we restrict to the Lie algebras $\lieg = A_r = \mathfrak{sl}_{r + 1}$, where it is well-known that crystal bases can be identified with appropriate semistandard Young tableaux.
The goal here is to connect various aspects of the construction of the higher-rank graph $\hrg$ with some well-known notions and operations in combinatorics related to such tableaux.

First we show how the Cartan braiding between two fundamental crystals can be computing using jeu de taquin slides on certain (skew) tableaux.
This gives a simple combinatorial procedure to compute the right ends of elements in a crystal $\calB(\lambda)$, which is shown to coincide with the (left) \emph{key} of these tableaux, a notion originally introduced by Lascoux and Schützenberger in \cite{keys}.
As a byproduct, we use known results on keys to show that the vertex set of $\hrg$ for $\lieg = A_r$ can be identified with the corresponding Weyl group $S_{r + 1}$.

\subsection{Tableaux in type $A$}

We assume familiarity with the basics of Young tableaux, for which we refer to the textbook \cite{fulton} for any unexplained definitions and results.
In the following we are going to refer to semistandard Young tableaux simply as \emph{tableaux}.

It is well known that the elements of connected crystals in type $A$ can be identified with such tableaux, a result originally shown in \cite[Theorem 3.4.2]{kashiwara-nakashima}. The precise correspondence is as follows.
Let $\lambda \in \dwei$ and write $\lambda = \omega_{i_1} + \cdots + \omega_{i_n}$ with $i_1 \leq \cdots \leq i_n$.
Then the connected crystal $\calB(\lambda)$ can be realized as the Cartan component of the tensor product $\calB(\omega_{i_1}) \otimes \cdots \otimes \calB(\omega_n)$.
In this way, an element $b \in \calB(\lambda)$ is mapped to a tableau $T$ via
\ytableausetup{boxsize = 1.4 em}
\begin{equation}
\label{eq:tableaux-identification}
b =
\ytableaushort{{t^1_1}, \vdots, {t^1_{i_1}}} \otimes \cdots \otimes \ytableaushort{{t^n_1}, \vdots, \vdots, {t^n_{i_n}}}
\ \mapsto \ 
T = \ytableaushort{{t^n_1} \cdots {t^1_1}, \vdots \cdots \vdots, \vdots \cdots {t^n_{i_1}}, \vdots \cdots, {t^n_{i_n}}} \ .
\end{equation}

Below we prove a variant of this result for the embedding of $\calB(\omega_i + \omega_j)$ into $\calB(\omega_i) \otimes \calB(\omega_j)$ with $i \geq j$. In this case, the elements can be identified with certain skew tableaux (a result surely known to experts, but for which we could not find an appropriate reference).

\begin{lemma}
\label{lem:Cartan-skew}
For $i \geq j$, the Cartan component of $\calB(\omega_i) \otimes \calB(\omega_j)$ can be identified with the set of skew tableaux with columns of lengths $(j, i)$ via
\[
\ytableaushort{{m_1}, \vdots, \vdots, {m_i}} \otimes \ytableaushort{{n_1}, \vdots, {n_j}}
\ \mapsto \
\ytableaushort{\none {m_1}, \none \vdots, {n_1} \vdots, \vdots \vdots, {n_j} {m_i}} \ .
\]
\end{lemma}

\begin{proof}
We make use of various results from \cite{bump-schilling}, but we should note that their conventions for the tensor product of crystals are opposite to ours.
Write $\calB_{\lambda / \mu}$ for the set of skew tableaux of skew shape $\lambda / \mu$. This set can be given the structure of an crystal by embedding it into a tensor power of $\calB(\omega_1)$ via an admissible reading, such as the \emph{column reading} (see \cite[Theorem 8.8]{bump-schilling} for the case of row reading, but the same is true for the column one).

The column reading maps a column with (strictly increasing) entries $i_1, \ldots, i_k$ to the element $\ytableaushort{{i_1}} \otimes \cdots \otimes \ytableaushort{{i_k}}$ of the tensor power $\calB(\omega_1)^{\otimes k}$ (corresponding to the fundamental crystal $\calB(\omega_k)$).
Next, consider a general skew shape $\lambda / \mu$, with columns $C_1, \ldots, C_n$ of lengths $(\ell_1, \ldots, \ell_n)$.
Then the column reading maps any tableau of this shape to $\calB(\omega_1)^{\otimes (\ell_1 + \cdots + \ell_n)}$ via
\[
\mathrm{CR}:
\ytableaushort{{C_1} \cdots {C_n}}
\mapsto 
\mathrm{CR}(\ytableaushort{{C_n}}) \otimes \cdots \otimes \mathrm{CR}(\ytableaushort{{C_1}}).
\]
Hence these tableaux can be seen as a subset of the tensor product $\calB(\omega_{\ell_n}) \otimes \cdots \otimes \calB(\omega_{\ell_1})$.

Applying this to our setting, we see that the skew tableaux with columns of lengths $(j, i)$ are contained in the tensor product $\calB(\omega_i) \otimes \calB(\omega_j)$.
Note that the highest weight element of the Cartan component is of this form, since the highest weight element of $\calB(\omega_k)$ is the column with entries $1, \ldots, k$.
To conclude the proof, we need to show that all such skew tableaux are obtained by the action of the Kashiwara operators on the highest weight element.

In general crystals corresponding to skew tableaux are not connected: according to \cite[Proposition 8.10]{bump-schilling} we have the decomposition $\calB_{\lambda / \mu} \cong \bigoplus_\nu \calB_\nu^{\oplus c_{\lambda \mu}^\nu}$, where the components $\calB_\nu$ are connected and the $c_{\lambda \mu}^\nu$ are the Littlewood-Richardson coefficients.
The latter can be characterized as the number of skew tableaux (with weight $\nu$) whose reading words are Yamanouchi words.
However, for the skew shape considered here, it is easy to see that there is a unique possible Yamanouchi word.
Then the result follows, as we have a connected crystal.
\end{proof}

\subsection{Right ends and jeu de taquin}

The jeu de taquin slides are certain moves which can be performed on a skew tableau to transform it into a tableau (the so-called \emph{rectification}), see for instance \cite[Section 1.2]{fulton}.
Keeping in mind the description from the previous subsection, we can realize the Cartan braiding as follows.

\begin{lemma}
\label{lem:cartan-slides}
Consider the Cartan braiding $\sigma: \calB(\omega_i) \otimes \calB(\omega_j) \to \calB(\omega_j) \otimes \calB(\omega_i)$ restricted to the the Cartan components of the two tensor products.
Then $\sigma(b \otimes b')$ can be obtained by applying jeu de taquin slides to the (skew) tableau corresponding to $b \otimes b'$.
\end{lemma}

\begin{proof}
Recall that the Cartan braiding is the unique isomorphism of crystals between the two Cartan components (and zero otherwise).
As such, it sends the highest weight element of the first tensor product to the highest weight element of the second tensor product.
Taking into account \cref{lem:Cartan-skew}, it is immediate to see this mapping is realized by jeu de taquin slides, pictured below for $i \geq j$ (with the case $i \leq j$ corresponding to this move in reverse):
\[
\ytableaushort{\none 1, \none \vdots, 1 \vdots, \vdots \vdots, j i}
\quad \xrightarrow{\mathrm{jeu \ de \ taquin}} \quad 
\ytableaushort{1 1, \vdots \vdots, \vdots j, \vdots, i} \ .
\]
It is proven in \cite[Theorem 3.3.1]{vanleeuwen} that the action of the Kashiwara operators commutes with the jeu de taquin slides (note that the Kashiwara operators are called \emph{coplactic operations} in this reference).
The result then follows, since all other elements can be obtained from the highest weight elements under the action of the Kashiwara operators.
\end{proof}

\begin{remark}
\cref{lem:cartan-slides} can probably be deduced from \cite{lenart-commutor}, which gives a combinatorial realization of the crystal commutor in terms of certain local moves (which coincides with the Cartan braiding on the Cartan component, as in \cref{rmk:commutor-cartan}).
However, as this reference does not explicitly work with tableaux, we prefer to proceed directly as in the lemma.
\end{remark}

We illustrate this result in the case of our running example.

\ytableausetup{boxsize = 1.1 em}

\begin{example}
Consider $\lieg = A_2$. We retain the notation for the crystal bases of the two fundamental crystals from \cref{sec:example-a2}.
They correspond to the following tableaux:
\[
\begin{gathered}
a_1 = \ytableaushort{1} \ , \quad
a_2 = \ytableaushort{2} \ , \quad
a_3 = \ytableaushort{3} \ , \\
b_1 = \ytableaushort{1, 2} \ , \quad
b_2 = \ytableaushort{1, 3} \ , \quad
b_3 = \ytableaushort{2, 3} \ .
\end{gathered}
\]
Consider the Cartan braiding $\calB(\omega_1) \otimes \calB(\omega_2) \to \calB(\omega_2) \otimes \calB(\omega_1)$, restricted to the Cartan components.
The relevant jeu de taquin slides in this case are pictured in \cref{fig:JDT-A2}.

\begin{figure}[h]
\begin{center}

\begin{tikzpicture}
\node at (0, 0) {$\ytableaushort{1 1, 2} \ \mapsto \ \ytableaushort{\none 1, 1 2} \ ,$};
\node at (4, 0) {$\ytableaushort{1 2, 2} \ \mapsto \ \ytableaushort{\none 1, 2 2} \ ,$};
\node at (8, 0) {$\ytableaushort{1 3, 2} \ \mapsto \ \ytableaushort{\none 1, 2 3} \ ,$};
\node at (0, -1.5) {$\ytableaushort{1 1, 3} \ \mapsto \ \ytableaushort{\none 1, 1 3} \ ,$};
\node at (4, -1.5) {$\ytableaushort{1 2, 3} \ \mapsto \ \ytableaushort{\none 2, 1 3} \ ,$};
\node at (8, -1.5) {$\ytableaushort{1 3, 3} \ \mapsto \ \ytableaushort{\none 1, 3 3} \ ,$};
\node at (4, -3) {$\ytableaushort{2 2, 3} \ \mapsto \ \ytableaushort{\none 2, 2 3} \ ,$};
\node at (8, -3) {$\ytableaushort{2 3, 3} \ \mapsto \ \ytableaushort{\none 2, 3 3} \ .$};
\end{tikzpicture}
\end{center}

\caption{Jeu de taquin slides for the Cartan component of $\calB(\omega_1) \otimes \calB(\omega_2)$.}
\label{fig:JDT-A2}
\end{figure}

Translating this back into elements of the crystal bases, we obtain
\[
\begin{array}{lll}
\sigma(a_1 \otimes b_1) = b_1 \otimes a_1, &
\sigma(a_2 \otimes b_1) = b_1 \otimes a_2, &
\sigma(a_3 \otimes b_1) = b_2 \otimes a_2, \\
\sigma(a_1 \otimes b_2) = b_2 \otimes a_1, &
\sigma(a_2 \otimes b_2) = b_3 \otimes a_1, &
\sigma(a_3 \otimes b_2) = b_2 \otimes a_3, \\
\sigma(a_1 \otimes b_3) = 0, &
\sigma(a_2 \otimes b_3) = b_3 \otimes a_2, &
\sigma(a_3 \otimes b_3) = b_3 \otimes a_3.
\end{array}
\]
This reproduces the direct computation appearing in \cite[Example 5.7]{matassa-yuncken}.
\end{example}

Next, we show how to use \cref{lem:cartan-slides} to compute the right ends of $\calB(\rho)$.

\begin{proposition}
\label{prop:rightend-slides}
Let $\lambda \in \dwei$ and write $\lambda = \omega_{i_1} + \cdots + \omega_{i_n}$ with $i_1 \leq \cdots \leq i_n$.
Let $b \in \calB(\lambda)$ and identify this element with a tableau $T$ as in \eqref{eq:tableaux-identification}.
Then we can compute the right ends $\left( \Rend_{i_1}(b), \ldots, \Rend_{i_n}(b) \right)$ of $b$ using the corresponding tableau $T$ as follows:
\begin{itemize}
\item apply jeu de taquin slides to $T$ to bring the column of length $i_k$ to the leftmost position,
\item obtain the right end $\Rend_{i_k}(b)$ as the leftmost column of this skew tableau.
\end{itemize}
\end{proposition}

\begin{proof}
Write $b = b_1 \otimes \cdots \otimes b_n$ with $b_k \in \calB(\omega_{i_k})$.
According to \cref{prop:computation-rightend}, the right end $\Rend_{i_k}(b)$ for $k = 1, \ldots, n$ can be determined as follows: we use the Cartan braiding $\sigma$ to compute $\sigma_{r - 1} \cdots \sigma_k (b_1 \otimes \cdots \otimes b_r)$ and obtain the result as the rightmost tensor factor.
Each of the Cartan braidings $\calB(\omega_i) \otimes \calB(\omega_j) \to \calB(\omega_j) \otimes \calB(\omega_i)$ can be computed in terms of jeu de taquin slides between two-column (skew) tableaux, as shown in \cref{lem:cartan-slides}.
Since the correspondence reverses products (with the conventions we are using for the tensor product of crystals), we see that the rightmost tensor factor of $\sigma_{r - 1} \cdots \sigma_k (b_1 \otimes \cdots \otimes b_r)$ corresponds to the leftmost column of the skew tableau obtained by applying the jeu de taquin slides.
\end{proof}

\begin{remark}
In particular, this result can be used to compute the right ends of $\calB(\rho)$.
In this case we have $\rho = \omega_1 + \cdots + \omega_r$ and $\Rend(b) = \left( \Rend_1(b), \ldots, \Rend_r(b) \right)$.
\end{remark}

We illustrate this procedure in our main example.

\begin{example}
We show how to use \cref{prop:rightend-slides} to determine the right ends of $\calB(\rho)$ for $\lieg = A_2$.
We have already exhibited the relevant jeu de taquin slides in \cref{fig:JDT-A2}.
Then taking the leftmost columns appearing in the various (skew) tableaux we get the following:

\begin{figure}[h]
\begin{center}

\ytableausetup{boxsize = 1.1 em}
\begin{tikzpicture}[scale=0.8]
\node at (0, 0) {$\Rend \left( \ytableaushort{1 1, 2} \right) = \ \ytableaushort{1 1, 2} \ ,$};
\node at (5.5, 0) {$\Rend \left( \ytableaushort{1 2, 2} \right) = \ \ytableaushort{1 2, 2} \ ,$};
\node at (11, 0) {$\Rend \left( \ytableaushort{1 3, 2} \right) = \ \ytableaushort{1 2, 2} \ ,$};
\node at (0, -2) {$\Rend \left( \ytableaushort{1 1, 3} \right) = \ \ytableaushort{1 1, 3} \ ,$};
\node at (5.5, -2) {$\Rend \left( \ytableaushort{1 2, 3} \right) = \ \ytableaushort{1 1, 3} \ ,$};
\node at (11, -2) {$\Rend \left( \ytableaushort{1 3, 3} \right) = \ \ytableaushort{1 3, 3} \ ,$};
\node at (5.5, -4) {$\Rend \left( \ytableaushort{2 2, 3} \right) = \ \ytableaushort{2 2, 3} \ ,$};
\node at (11, -4) {$\Rend \left( \ytableaushort{2 3, 3} \right) = \ \ytableaushort{2 3, 3} \ .$};
\end{tikzpicture}
\end{center}

\caption{The right ends of the various elements of $\calB(\rho)$ for $\lieg = A_2$.}
\end{figure}

Rewriting this in terms of the previously introduced notation, we get
\[
\begin{array}{lll}
\Rend(a_1 \otimes b_1) = (a_1, b_1), &
\Rend(a_2 \otimes b_1) = (a_2, b_1), &
\Rend(a_3 \otimes b_1) = (a_2, b_1), \\
\Rend(a_1 \otimes b_2) = (a_1, b_2), &
\Rend(a_2 \otimes b_2) = (a_1, b_2), &
\Rend(a_3 \otimes b_2) = (a_3, b_2), \\
&
\Rend(a_2 \otimes b_3) = (a_2, b_3), &
\Rend(a_3 \otimes b_3) = (a_3, b_3).
\end{array}
\]
This reproduces the direct computation in \cite[Example 7.9]{matassa-yuncken}.
We note that the six distinct right ends can be identified with the Weyl group elements.
\end{example}

\subsection{Keys of tableaux}

The notion of key of a tableau was introduced by Lascoux and Schützenberger in \cite{keys} to analyze the combinatorics of standard bases of sections of line bundles (see also \cite[Appendix A.5]{fulton}).
We review this notion below.

A skew tableau is called \emph{frank} if the lengths of its columns are the same (up to a permutation) as the lengths of the columns of its rectification.
Let $T$ be a tableau with $n$ columns of lengths $(\ell_1, \ldots, \ell_n)$. For each column $1 \leq i \leq n$, choose a frank skew tableau $S_i$ such that its rightmost column has length $\ell_i$. We denote such a column by $R_i(T)$, since it turns out not to depend on the specific choice of $S_i$. Then the \emph{right key} $\rkey(T)$ of $T$ is the tableau of the same shape as $T$ and with columns given by $R_1(T), \ldots, R_n(T)$.

A similar definition is given for the \emph{left key} $\lkey(T)$.
We consider frank skew tableaux whose leftmost columns $L_i(T)$ have lengths $\ell_i$, and form the key by the columns $L_1(T), \ldots, L_n(T)$.

In practice, the keys can be computed by applying jeu de taquin slides to the tableau $T$ to produce the required skew tableaux, as illustrated in the next example.

\begin{example}
Let $T$ be the tableau appearing to the left of \cref{fig:example-keys}, having columns of lengths $(3, 2, 1)$.
We can use jeu de taquin slides to produce the (frank) skew tableaux with columns of the required lengths, either to the left or to the right, as depicted in the figure.

\begin{figure}[h]
\begin{center}

\begin{tikzpicture}[scale=0.8]
\node (t1) at (0, 0) {$T = \ytableaushort{1 2 3, 2 5, 4}$};

\node (t2) at (4, 2) {$\ytableaushort{\none 1 3, 2 2, 4 5}$};

\node (t3) at (8, 2) {$\ytableaushort{\none \none 1, \none \none 3, 2 2 5, 4}$};

\node (t4) at (12, 0) {$\ytableaushort{\none \none 1, \none 2 3, 2 4 5}$};

\node (t5) at (4, -2) {$\ytableaushort{\none \none 3, 1 2 5, 2, 4}$};

\node (t6) at (8, -2) {$\ytableaushort{\none 1 3, \none 2 5, 2 4}$};

\draw (t1) -- (t2);
\draw (t2) -- (t3);
\draw (t3) -- (t4);
\draw (t1) -- (t5);
\draw (t5) -- (t6);
\draw (t6) -- (t4);
\end{tikzpicture}
\end{center}

\caption{Jeu de taquin slides applied to the starting tableau $T$.}
\label{fig:example-keys}
\end{figure}

Then the moves from \cref{fig:example-keys} show that the left and right keys of $T$ are given by
\[
\lkey(T) = \ytableaushort{1 2 2, 2 4, 4} \ , \qquad
\rkey(T) = \ytableaushort{1 3 3, 3 5, 5} \ .
\]
\end{example}

The next result shows that, in type $A$, the notion of right end essentially coincides with that of left key (the left-right exchange being purely a matter of convention, see \cref{rmk:key-swap}).

\begin{theorem}
\label{thm:keys-ends}
Let $\lieg = A_r$.
Let $\lambda \in \dwei$ and write $\lambda = \omega_{i_1} + \cdots + \omega_{i_n}$ with $i_1 \leq \cdots \leq i_n$.
Let $b \in \calB(\lambda)$ and identify this element with a tableau $T$ as in \eqref{eq:tableaux-identification}.
Then the right ends $\left( \Rend_{i_1}(b), \ldots, \Rend_{i_n}(b) \right)$ can be identified with the columns of the left key $\lkey(T)$.
\end{theorem}

\begin{proof}
By definition of the left key $\lkey(T)$, its columns are given by the leftmost columns $L_k(T)$ of the appropriate frank skew tableaux $S_k'$ corresponding to the tableau $T$.
As shown in \cite[Appendix A.5]{fulton}), such frank skew tableaux can be obtained by applying jeu de taquin slides to adjacent columns of $T$.
By \cref{lem:cartan-slides} we know that any such move between adjacent columns corresponds to the action of the Cartan braiding on a tensor product.
Recall that the correspondence between tableaux and Cartan components used in \cref{prop:rightend-slides} reverses the order of products.
Then we see that the leftmost column $L_k(T)$ of the skew tableau $S_k'$ corresponds to the rightmost factor of the corresponding tensor product.
By \cref{prop:rightend-slides} this corresponds to the right end $\Rend_{i_k}(b)$, which proves the result.
\end{proof}

\begin{remark}
\label{rmk:key-swap}
The right-to-left exchange in \cref{thm:keys-ends} is entirely due to the reversal of products in the correspondence \eqref{eq:tableaux-identification}.
Note that we could have chosen to work with the opposite convention for the tensor product of crystals, as done in \cite{bump-schilling} (which is better adapted to describe classical results in combinatorics).
With this different convention, the right ends would correspond to the right keys, which is a bit more natural.
\end{remark}

An an easy corollary of this identification, we obtain the following result.

\begin{corollary}
The vertex set of the higher-rank graph $\hrg$ for $\lieg = A_r$ can be identified with the corresponding Weyl group $W = S_{r + 1}$.
\end{corollary}

\begin{proof}
It is well-known that the image of the (left or right) key map on tableaux with entries in $\{ 1, \cdots, n \}$ can be seen as a subset of the symmetric group $S_n$ (for instance, in \cite[Section 2.6]{hersh-lenart} a generalization of the key map is defined directly with values in the Weyl group).
\cref{thm:keys-ends} allows us to compute the right ends of $\calB(\rho)$ for $\lieg = A_r$ in terms of the left key map. Hence the vertex set of $\hrg$ can be identified with a subset of the Weyl group $W = S_{r + 1}$.
But the latter embeds in the vertex set of $\hrg$ by \cref{prop:weyl-vertex}, which shows equality.
\end{proof}

It should also be possible to interpret the paths in the higher-rank graph $\hrg$ in a similar way, by considering products of key tableaux and general tableaux as in \cite[Lemma 3]{fulton-pieri}.
However we are not going to investigate this connection here.

\section{Keys beyond type $A$}
\label{sec:beyond-A}

The notion of key can be generalized to a map $\kappa: \calB \to W$, where $\calB$ is a crystal of an arbitrary symmetrizable Kac-Moody algebra $\lieg$ and $W$ its corresponding Weyl group, see for instance \cite{hersh-lenart} and references therein.
This generalization is based on the notion of \emph{Lakshmibai--Seshadri paths} introduced by Littelmann in \cite{littelmann}: the right and left keys can be defined using such paths in terms of their initial and final directions, respectively.
An algorithm to compute the (right) key, without reference to the path model, is given \cite[Theorem 3.4]{hersh-lenart}.
It should be mentioned that other methods are available to compute such keys, but the one the cited paper appears more closely related to the setting considered here.

It is known that such (generalized) keys can be identified with the keys of Lascoux and Schützenberger for $\lieg$ of type $A$.
Then it follows from \cref{thm:keys-ends} that, in this case, the right ends can also be identified with such keys.
On the other hand, it is not clear how one should compare the two notions beyond type $A$.
The two notions are certainly distinct in general, since the key as defined above takes value in the Weyl group $W$, while the right ends do not necessarily correspond to Weyl group elements (see for instance the example below).

The rest of this section is dedicated to the computation of the right ends and the keys corresponding to the crystal $\calB(\rho)$ for the Lie algebra $\lieg = C_2$.
This is simply meant as an invitation for further study on the connection between these two notions.

\emph{Warning}: in the example below we switch to the \emph{opposite} conventions for the tensor product of crystals, since we want to use the rules given in \cite[Theorem 3.4]{hersh-lenart}, which allow the computation of the right keys.
We note that, with these opposite conventions for the tensor product, \cref{thm:keys-ends} gives an identification between right ends and right keys instead.

\begin{example}
We consider $\lieg = C_2$, whose right ends have been computed in \cite[Example 10.6]{matassa-yuncken} with different conventions.
The two fundamental crystals are
\[
\begin{split}
\calB(\omega_1) : \quad & a_1 \xrightarrow{1} a_2 \xrightarrow{2} a_3 \xrightarrow{1} a_4, \\
\calB(\omega_2) : \quad & b_1 \xrightarrow{2} b_2 \xrightarrow{1} b_3 \xrightarrow{1} b_4 \xrightarrow{2} b_5.
\end{split}
\]
The crystal $\calB(\rho)$ can be identified with the Cartan component of $\calB(\omega_1) \otimes \calB(\omega_2)$.
The crystal graph corresponding to this component is given by
\begin{center}
\begin{tikzcd}[column sep=1em, row sep=1em]
a_1 \otimes b_1 \arrow[r, blue, "2"] \arrow[d, red, "1"] & a_1 \otimes b_2 \arrow[r, red, "1"] & a_1 \otimes b_3 \arrow[r, red, "1"] & a_1 \otimes b_4 \arrow[r, blue, "2"] \arrow[d, red, "1"] & a_1 \otimes b_5 \arrow[d, red, "1"] \\
a_2 \otimes b_1 \arrow[r, blue, "2"] & a_2 \otimes b_2 \arrow[r, red, "1"] \arrow[d, blue, "2"] & a_2 \otimes b_3 \arrow[d, blue, "2"] & a_2 \otimes b_4 \arrow[r, blue, "2"] & a_2 \otimes b_5 \arrow[d, blue, "2"] \\
& a_3 \otimes b_2 \arrow[r, red, "1"] & a_3 \otimes b_3 \arrow[r, red, "1"] & a_3 \otimes b_4 \arrow[d, red, "1"] & a_3 \otimes b_5 \arrow[d, red, "1"] \\
& & & a_4 \otimes b_4 \arrow[r, blue, "2"] & a_4 \otimes b_5
\end{tikzcd}
\end{center}
(Note that this corresponds to the opposite of the graph for $\calB(\omega_2) \otimes \calB(\omega_1)$ from \cite[Example 10.6]{matassa-yuncken}, due to the conventions used here).
The Cartan braiding is given by
\begin{small}
\begin{align*}
\sigma(a_1 \otimes b_1) & = b_1 \otimes a_1, &
\sigma(a_1 \otimes b_2) & = b_2 \otimes a_1, &
\sigma(a_1 \otimes b_3) & = b_2 \otimes a_2, &
\sigma(a_1 \otimes b_4) & = b_3 \otimes a_2, \\
\sigma(a_1 \otimes b_5) & = b_3 \otimes a_3, &
\sigma(a_2 \otimes b_1) & = b_1 \otimes a_2, &
\sigma(a_2 \otimes b_2) & = b_1 \otimes a_3, &
\sigma(a_2 \otimes b_3) & = b_1 \otimes a_4, \\
\sigma(a_2 \otimes b_4) & = b_4 \otimes a_2, &
\sigma(a_2 \otimes b_5) & = b_4 \otimes a_3, &
\sigma(a_3 \otimes b_2) & = b_2 \otimes a_3, &
\sigma(a_3 \otimes b_3) & = b_2 \otimes a_4, \\
\sigma(a_3 \otimes b_4) & = b_3 \otimes a_4, &
\sigma(a_3 \otimes b_5) & = b_5 \otimes a_3, &
\sigma(a_4 \otimes b_4) & = b_4 \otimes a_4, &
\sigma(a_4 \otimes b_5) & = b_5 \otimes a_4.
\end{align*}
\end{small}
Using this data, we can easily compute the right ends for the crystal $\calB(\rho)$:
\begin{center}
\begin{tikzcd}[column sep=0.1em, row sep=0.1em]
(a_1, b_1) & (a_1, b_2) & (a_2, b_3) & (a_2, b_4) & (a_3, b_5) \\
(a_2, b_1) & (a_3, b_2) & (a_4, b_3) & (a_2, b_4) & (a_3, b_5) \\
& (a_3, b_2) & (a_4, b_3) & (a_4, b_4) & (a_3, b_5) \\
& & & (a_4, b_4) & (a_4, b_5)
\end{tikzcd}
\end{center}
These are arranged in the same pattern as the crystal graph of $\calB(\rho)$ given above.
Some of these right ends can be put in correspondence with the Weyl group elements of $C_2$, as per the general \cref{prop:weyl-vertex}.
We list these (with the same pattern) in \cref{fig:weyl-C2}, and write $*$ for those vertices which do not corresponding to Weyl group elements.

\begin{figure}[h]
\begin{center}

\begin{tikzcd}[column sep=0.1em, row sep=0.1em]
1 & s_2 & * & s_1 s_2 & s_2 s_1 s_2 \\
s_1 & s_2 s_1 & * & s_1 s_2 & s_2 s_1 s_2 \\
& s_2 s_1 & * & s_1 s_2 s_1 & s_2 s_1 s_2 \\
& & & s_1 s_2 s_1 & s_1 s_2 s_1 s_2
\end{tikzcd}
\end{center}

\caption{The Weyl group elements corresponding to the right ends of $\calB(\rho)$.}
\label{fig:weyl-C2}
\end{figure}

On the other hand, the right keys for the crystal $\calB(\rho)$ can be computed in terms of the rules given in \cite[Theorem 3.4]{hersh-lenart}.
We refer to the cited paper for the precise statement of these rules, and simply report the result of the computation in \cref{fig:keys-C2}.

\begin{figure}[h]
\begin{center}

\begin{tikzcd}[column sep=0.1em, row sep=0.1em]
1 & s_2 & s_1 s_2 & s_1 s_2 & s_2 s_1 s_2 \\
s_1 & s_2 s_1 & s_1 s_2 s_1 & s_1 s_2 & s_2 s_1 s_2 \\
& s_2 s_1 & s_1 s_2 s_1 & s_1 s_2 s_1 & s_2 s_1 s_2 \\
& & & s_1 s_2 s_1 & s_1 s_2 s_1 s_2
\end{tikzcd}
\end{center}

\caption{The right keys for the crystal $\calB(\rho)$.}
\label{fig:keys-C2}
\end{figure}

Comparing \cref{fig:weyl-C2} and \cref{fig:keys-C2}, we see that they coincide for all columns except one, which is precisely where the right ends cannot be identified with any Weyl group elements.
Summarizing, in the example of $\lieg = C_2$ we have that the right ends and the right keys of crystal elements coincide, whenever the former can be identified with Weyl group elements.
\end{example}

The example discussed above motivates the following question.

\begin{question}
Consider the crystal $\calB(\rho)$ corresponding to a finite-dimensional semisimple Lie algebra $\lieg$.
Suppose that the right end of $b \in \calB(\rho)$ corresponds to a Weyl group element $w \in W$. Do we have $\kappa(b) = w$ for the right key?
What about for a general crystal $\calB(\lambda)$?
\end{question}

\end{document}